\documentclass[ a4paper, 12pt, reqno ]{amsart}
\usepackage{ amsmath, amssymb, amsthm, fancyhdr, enumerate, enumitem, tabu, units }
\usepackage[margin=1 in,bottom=1.5 in]{geometry}
\usepackage[toc,page]{appendix}
\usepackage{biblatex}
\usepackage{xfrac}
\usepackage{xcolor}
\usepackage{lscape}

\newcommand{\QQ}{\mathbb{Q}}
\newcommand{\ZZ}{\mathbb{Z}}

\newtheorem{theorem}{Theorem}[section]

\newtheorem{lemma}[theorem]{Lemma}

\newtheorem{prop}[theorem]{Proposition}

\theoremstyle{definition}

\newtheorem{remark}{Remark}
\newtheorem{defn}{Definition}

\begin{document}
\title{Existence of $\eta$-quotients for Squarefree Levels}
\author{Michael Allen}
\date{\today}
\begin{abstract}
  In this paper, we investigate which modular form spaces can contain $\eta$-quotients, functions of the form $f(\tau) = \prod_{\delta \mid N} \eta(\delta\tau)^{r_\delta}$.  For $k\geq 2$ even and $N$ coprime to 6, we give necessary conditions for the space $M_k(\Gamma_1(N))$ to contain $\eta$-quotients.  We then show that these conditions are sufficient as well for a large family of squarefree levels $N$ coprime to 6.
\end{abstract}
\maketitle
\section{Introduction and Statement of Results}
Dedekind's $\eta$ function, defined by
\[\eta(\tau) = q^\frac{1}{24}\prod_{n=1}^\infty (1-q^n),\]
where $\tau$ is in the upper-half of the complex plane and $q := e^{2\pi i \tau}$, is one of the most famous and classical examples of a weight $\frac{1}{2}$ modular form.  The infinite product formula makes $\eta(\tau)$ particularly nice to work with for many computations.  As such, it is interesting to know when a modular form can be written as a linear combination of products and quotients of $\eta$-functions.  An $\eta$-quotient is a function of the form 
\begin{equation*}
 f(\tau) = \prod_{\delta \mid N} \eta(\delta\tau)^{r_\delta},
\end{equation*}
where each $r_\delta \in \ZZ$.  \\

Due to the nice form of $\eta(\tau)$, a question of interest is which modular objects can be expressed as an $\eta$-quotient or a linear combination of $\eta$-quotients.  In \cite{WebOf}, Ono shows that every holomorphic modular form for the entire modular group $\text{SL}_2(\ZZ)$ can be written as a rational function of $\eta(z)$, $\eta(2z)$, and $\eta(4z)$.  Martin and Ono \cite{MO} have classified all elliptic curves which are represented by $\eta$-quotients via the modularity theorem.  Lemke Oliver \cite{LO} has classified all $\theta$-functions which are $\eta$-quotients.  Additional work on spaces of modular forms spanned by $\eta$-quotients has been done by Arnold-Roksandich, James, and Keaton \cite{ARJK} and Rouse and Webb \cite{RW}, for example.  In \cite{AAHOS}, the author et al. investigate the question of which modular form spaces cannot contain $\eta$-quotients; these results are extended in this paper. \\

Throughout the paper we use the following notation.  Given a subgroup $\Gamma \subset \text{SL}_2(\ZZ)$, we use $S_k(\Gamma)$, $M_k(\Gamma)$, and $M_k^!(\Gamma)$ to denote the complex vector spaces of weight $k$ cusp forms, holomorphic modular forms, and weakly holomorphic forms, respectively, for the group $\Gamma$.  We use $M(\Gamma)$ to denote the graded ring of all holomorphic modular forms for $\Gamma$.  In the case where $\Gamma = \Gamma_0(N)$, we use $S_k(\Gamma_0(N), \chi), M_k(\Gamma_0(N), \chi)$, and $M^!_k(\Gamma_0(N), \chi)$ to denote the spaces of weight $k$ cusp forms, holomorphic modular forms, and weakly holomorphic forms, respectively, for the group $\Gamma_0(N)$ with Nebentypus $\chi$. \\

The following theorem, originally appearing in work of Newman \cite{Newman,Newman2} as well as Gordon and Hughes \cite{GH}, is useful for determining when a given $\eta$-quotient is a weakly holomorpic modular form.
\begin{theorem}[Theorem 1.64, \cite{WebOf}]\label{1.64}
Let $f$ be the $\eta$-quotient of level $N$ given by
\[f(\tau) = \prod_{\delta \mid N} \eta(\delta\tau)^{r_\delta}. \]
If $f$ satisfies both
\begin{align*}
 \sum_{\delta \mid N} \delta r_\delta &\equiv 0 \pmod{24} \\
 \sum_{\delta \mid N} \frac{N}{\delta}r_\delta &\equiv 0 \pmod{24},
\end{align*}
then for $k = \frac{1}{2}\sum_{\delta \mid N} r_\delta$ and $\chi(d) = \left(\frac{(-1)^ks}{d}\right)$, where $s = \prod_{\delta \mid N} \delta^{r_\delta}$, we have $f \in M_k^!(\Gamma_0(N),\chi)$.
\end{theorem}

\begin{remark}\label{1.64conv}
In the case where $N$ is coprime to 6, the two congruences in Theorem \ref{1.64} are equivalent, as any number coprime to 6 is its own inverse modulo 24.  Additionally, the converse to Theorem \ref{1.64} holds when $N$ is coprime to 6.  Moreover, any weakly holomorphic $\eta$-quotient on $\Gamma_1(N)$ must also satisfy the hypotheses of Theorem \ref{1.64} (see Corollary 1.8 of \cite{AAHOS}).  That is, given $f(\tau) = \prod_{\delta \mid N} \eta(\delta\tau)^{r_\delta} \in M_k^!(\Gamma_1(N))$, we have
\[\sum_{\delta \mid N} \delta r_\delta \equiv 0 \pmod{24},\]
and hence $f(\tau) \in M_k^!(\Gamma_0(N),\chi)$ for some Nebentypus $\chi$.  Thus, every $\eta$-quotient which is modular for $\Gamma_1(N)$ is in fact modular for $\Gamma_0(N)$ whenever $\gcd(N,6) = 1$.  Because of this, in working with $\eta$-quotients on $\Gamma_1(N)$, we only need to consider the behavior at the cusps of $\Gamma_0(N)$.
\end{remark}

In work of Martin \cite{Martin}, the following is shown to be a complete set of representatives for the cusps of $\Gamma_0(N)$:
\[C_{\Gamma_0(N)} = \left\{\frac{a_c}{c} \in \QQ \hspace{2mm} \middle\vert \begin{array}{l}
c \mid N, 1 \leq a_c \leq N, \gcd(a_c,N) = 1
\\ \text{and } a_c \equiv a_c' \pmod{\gcd\left(c,\frac{N}{c}\right)} \iff a_c = a_c' \end{array}\right\}.\]
When $N$ is squarefree, as will be the case throughout this paper, this can be reduced to
\begin{equation}\label{cuspreps}
C_{\Gamma_0(N)} = \left\{\frac{1}{d} : d \mid N \right\}.
\end{equation}

In order to determine when $\eta$-quotients are fully holomorphic modular forms, we need to compute orders of vanishing at the cusps.  We do so using a theorem of Ligozat \cite{Ligozat}, also appearing in work of Biagioli \cite{Biagioli} and Martin \cite{Martin}.
\begin{theorem}[Theorem 1.65, \cite{WebOf}]\label{1.65}
N Let $c,d$, and $N$ be positive integers with $d \mid N$ and $\gcd(c,d) = 1$.  Then if $f(\tau)$ is an $\eta$-quotient satisfying the conditions given in Theorem \ref{1.64} for $N$, then the order of vanishing for $f(\tau)$ at the cusp $c/d$ relative to $\Gamma_0(N)$ is
\[v_{\frac{c}{d}} := v_{\Gamma_0(N)}\left(f,\frac{c}{d}\right) = \frac{N}{24} \sum_{\delta \mid N} \frac{\gcd(d,\delta)^2r_\delta}{\gcd(d,\frac{N}{d})d\delta}.\]
\end{theorem}

The next theorem, first appearing in work of Rouse and Webb \cite{RW}, says that for $M(\Gamma_0(N))$ to be generated by $\eta$-quotients, $N$ must be sufficiently composite relative to its size.
\begin{theorem}[\cite{RW}]\label{RWB}
Suppose that $f(\tau) = \prod_{\delta \mid N} \eta(\delta\tau)^{r_\delta} \in M_k(\Gamma_0(N),
\chi)$.  Then we have
\[\sum_{\delta \mid N} |r_\delta| \leq 2k \prod_{p \mid N} \left(\frac{p+1}{p-1}\right)^{\min\left\{2,\text{ord}_p(N)\right\}}.\]
\end{theorem}

This bound implies the existence of an upper bound on the number of $\eta$-quotients in any space $M_k(\Gamma_0(N),\chi)$ which depends only on the number of prime divisors of $N$.  In particular, as $\text{dim}(M_k(\Gamma_0(N),\chi))$ tends towards infinity as $N$ does, there can only exist a finite number of levels with a given number of prime divisors such that the ring $M(\Gamma_0(N))$ is generated by $\eta$-quotients. \\

In \cite{AAHOS}, the author et al. investigate other factors in addition to Theorem \ref{RWB} which inhibit the existence of $\eta$-quotients in certain modular form spaces.  The first result from \cite{AAHOS} gives a complete categorization of which space $M_k(p)$ contain $\eta$-quotients when $k\geq 2$ is even and $p > 5$ is prime.
\begin{theorem}[Theorem 1.9 \cite{AAHOS}]\label{primeexist}
Let $p$ be prime, set $h_p = \frac{1}{2} \gcd(p-1,24)$, and let $k$ be an even integer.  Then there exists $f(\tau) = \eta^{r_1}(\tau)\eta^{r_p}(p\tau) \in M_k(\Gamma_1(p))$ if and only if both of the following conditions hold.
\begin{enumerate}[label=(\roman*)]
 \item $h_p \mid k$
 \item It is not the case that $k=2$, $p \equiv 5 \pmod{24}$, and $p \ne 5$.
\end{enumerate}
\end{theorem}
The next theorem gives a similar classification for spaces $M_k(\Gamma_1(pq))$, where again $k \geq 2$ is even and $p,q > 5$ are distinct primes.
\begin{theorem}[Theorem 1.10 \cite{AAHOS}]\label{semiexist}
Let $p,q$ be distinct primes and let $N = pq$.  Let $k\geq 2$ be even, and define $h_N := \frac{1}{2}\gcd\left\{p-1,q-1,24\right\}.$  There exists an $\eta$-quotient $f(\tau) \in M_k(\Gamma_1(N))$ if and only if
\begin{enumerate}[label=(\roman*)]
 \item $h_N \mid k$
 \item It is not the case that $k=2$, $(p,q) \equiv (1,5), (5,1),$ or $(5,5) \pmod{24}$, and neither $p$ nor $q$ is equal to 5.
\end{enumerate}
\end{theorem}
These results generalize to squarefree levels as follows.  First, define
\begin{equation*}
S:= \left\{ \prod_{i=1}^\ell p_i^{e_i} \middle\vert \begin{array}{l}
\text{For all } i, p_i > 5 \text{ prime, } p_i \equiv 1 \text{ or } 5 \pmod{24} \\
\text{and at least one } p_i \equiv 5 \pmod{24} \end{array} \right\}.
\end{equation*}
Note that with this definition, condition (ii) in Theorem \ref{primeexist} (resp. Theorem \ref{semiexist}) can be rephrased as ``It is not the case that $k=2$ and $p \in S$ (resp. $pq \in S$)''.  Our first result generalizes the forwards direction of Theorem \ref{semiexist} to arbitrary levels coprime to 6.
\begin{theorem}\label{exist}
Let $N=p_1^{e_1}p_2^{e_2}\cdots p_{\ell}^{e_\ell}$ be an integer coprime to 6, let $k \geq 2$ be even, and define $h_N = \frac{1}{2}\gcd\left\{p_1-1,p_2-1,\hdots,p_\ell-1,24\right\}$.  There exists an $\eta$-quotient in $M_k(\Gamma_1(N))$ if
\begin{enumerate}[label=(\roman*)]
    \item $h_N \mid k$
    \item It is not the case that $k=2$ and $N \in S$.
\end{enumerate}
\end{theorem}
Theorems \ref{primeexist} and \ref{semiexist} suggest that $M_2(\Gamma_1(N))$ cannot contain $\eta$-quotients if the prime divisors of $N$ fall into certain residue classes modulo 24.  On the other hand, Theorem \ref{RWB} tells us that the more composite a number is, the easier it is to have $\eta$-quotients of that level.  This complicates proving the converse of Theorem \ref{exist}, as the bounding term in Theorem \ref{RWB} can be made arbitrarily large as $N$ ranges through all squarefree integers.  However, we are able to give the following partial converse to Theorem \ref{exist}.
\begin{theorem}\label{sqrfrnonexist}
 Let $N\in S$ be squarefree and coprime to 6, and let $p_1$ denote the smallest prime dividing $N$.  If
 \begin{equation}\label{RWHyp}
  4\prod_{p \mid N} \frac{p+1}{p-1} < p_1 +1,
 \end{equation}
 then there is no $\eta$-quotient in $M_2(\Gamma_1(N))$.
\end{theorem}
The product in the left-hand side of (\ref{RWHyp}) is made larger the smaller the prime divisors of $N$ are.  Thus, the smallest number $N$ for which this extra hypothesis fails is
\[N = \!\!\!\!\!\!\!\!\!\!\!\!\prod_{\substack{\text{prime }p \leq M \\ p \equiv 1 \text{ or } 5 \text{ (mod }24)}} \!\!\!\!\!\!\!\!\!\!\!\! p,\]
where
\[M := \inf\left\{x \quad \middle\vert \quad 4 \times \!\!\!\!\!\!\!\!\!\!\!\!\prod_{\substack{\text{prime }p \leq x \\ p \equiv 1 \text{ or } 5 \text{ (mod }24)}} \!\!\!\!\!\!\!\!\! \left(\frac{p+1}{p-1}\right) \geq 30 \right\}.\]
The product
\[\prod_{\substack{\text{prime }p \leq x \\ p \equiv 1 \text{ or } 5 \text{ (mod }24)}} \!\!\!\!\!\!\!\!\! \left(\frac{p+1}{p-1}\right)\]
grows numerically like $\log\log\log(x)$, and so $M$ will be extremely large.  As $N$ is then the product of all primes congruent to 1 or 5 (mod 24) up to $M$, $N$ must be astronomically large before this hypothesis can fail.  We do not claim that the converse is false when $N$ becomes large enough for (\ref{RWHyp}) to fail, only that different techniques from ours would be needed to address this case. \\

The remainder of this paper is devoted to proving these two theorems.  Theorem \ref{exist} is handled in Section 2, and Section 3 deals with the proof of Theorem \ref{sqrfrnonexist}.

\section{Modular form spaces which contain $\eta$-quotients}
In order for there to exist $\eta$-quotients which are fully holomorphic modular forms for $\Gamma_1(N)$ and weight $k$, there certainly must at least exist weakly holomorphic $\eta$-quotients.  In \cite{AAHOS}, the authors et al. obtained the following result for existence of weakly holomorphic $\eta$-quotients.
\begin{theorem}[Theorem 4.1, \cite{AAHOS}]\label{weakexist}
Let $N=p_1^{e_1}p_2^{e_2}\cdots p_{\ell}^{e_\ell}$ with each $p_i \geq 5$, and define $h_N = \frac{1}{2}\gcd(p_1-1,\cdots,p_\ell-1,24)$.  There exists an $\eta$-quotient $f(\tau)= \prod_{\delta \mid N} \eta^{r_\delta} (\delta\tau) \in M_K^!(\Gamma_1(N))$ if and only if $h_N \mid k$.
\end{theorem}
\begin{remark}
In \cite{AAHOS}, the above theorem is only stated to hold for $N$ squarefree.  However, the proof makes no use of this assumption, and so still holds for this generalized version.
\end{remark}
Before proving Theorem \ref{exist} in full generality, we will prove it for the specialized case where $N$ is squarefree.  
We begin with a lemma which allows us to reduce this case to Theorem \ref{semiexist}.
\begin{lemma}\label{reduce}
Let $N = p_1p_2\cdots p_{\ell}$ with $\ell \geq 2$.  Then there exist distinct primes $p$ and $q$ dividing $N$ such that $h_N = h_{pq}$.  Moreover, if $N \notin S$, then $p$ and $q$ can be chosen such that $pq \notin S$.
\end{lemma}
\begin{proof}
We begin by setting some notation.  For an integer $a$, let $\tilde{a}$ denote the least residue of $a$ mod 24.  Let $R_N$ be the set of least residues of $N$'s prime divisors mod 24.  That is,
\[R_N = \left\{ \tilde{p_1},\tilde{p_2},\hdots,\tilde{p_\ell}\right\}.\]
By definition, $p_1 = \tilde{p_1} + 24m$ for some integer $m$.  Observe that if $d \mid p_1-1$ and $d \mid 24$, then $d \mid (p_1-1 - 24m) = \tilde{p_1}-1$.  Similarly, any divisor of both $\tilde{p_1}-1$ and $24$ must also divide $p_1-1$, and so in particular,
\[h_N = \frac{1}{2} \gcd\left\{p_1-1,p_2-1,\cdots,p_\ell-1,24\right\} = \frac{1}{2}\gcd\left\{\tilde{p_1}-1,\tilde{p_2}-1, \cdots, \tilde{p_\ell}-1, 24\right\}.\]
That is, $h_N$ depends only on $R_N$, the residues of the prime divisors of $N$ modulo 24.  We proceed to the proof by cases on the value of $h_N$.  As $h_N$ is one half an even divisor of 24, the possible values of $h_N$ are 12, 6, 4, 3, 2, and 1.  Note that if $N \in S$, then $R_N = \left\{5\right\}$ or $\left\{1,5\right\}$, and so $h_N = 2$.  Thus, in all cases excluding $h_N = 2$, we may assume that $N \notin S$. \\

If $h_N = 12$, we equivalently have that every prime divisor of $N$ is 1 (mod 24), and so $R_N = \left\{1\right\}$.  Letting $p$ and $q$ be any divisors of $N$, we again have $R_{pq} = \left\{1\right\}$, and so $h_{pq} = h_N$. Additionally, as $p \equiv q \equiv 1$ (mod 24), $pq \notin S$.  \\

Next, if $h_N = 6$, then every prime divisor is congruent to 1 (mod 12), and so $R_N \subseteq \left\{1,13\right\}$.  However, $R_N$ must contain 13, as otherwise we would have $h_N = 12$. Let $p \mid N$ be such that $p \equiv 13$ (mod 24), and let $q$ be any other prime divisor of $N$.  Then $R_{pq}$ is either $\left\{1,13\right\}$ or $\left\{13\right\}$ depending on the residue class of $q$, but in either case $h_{pq} = 6 = h_N$.  As $p \equiv 13$ (mod 24), we also have that $pq \notin S$, and so gives the desired divisor.  The case $h_N = 4$ is very similar.  For this to hold, we must have $R_N = \left\{1,17\right\}$ or $\left\{17\right\}$, and in either case so long as we choose $p \equiv 17$ (mod 24) then any choice of $q$ will suffice.  \\

If $h_N = 3$, then every prime divisor of $N$ is 1 (mod 6), and so $R_N \subseteq \left\{1,7,13,19\right\}$.  But $R_N$ must contain at least one of $7$ or $19$, as otherwise we would have $h_N = 6$ or 12.  We may thus pick some prime divisor $p$ of $N$ which belongs to one of these two residue classes.  Let $q$ be any other prime divisor of $N$.  The possible residue sets for $pq$ are 
\[\left\{7\right\}, \left\{19\right\}, \left\{1,7\right\}, \left\{7,13\right\}, \left\{7,19\right\}, \left\{1,19\right\}, \text{ and } \left\{13,19\right\}.\]  
One can quickly check that for each of these possibilities, $h_{pq} = 3$.  Additionally, as $p$ is not congruent to $1$ or $5$ (mod 24), it follows that $pq \notin S$.  \\

For $h_N = 2$, we have a similar situation.  The residue set must satisfy $R_N \subseteq \left\{1,5,13,17\right\}$, and either $5$ or $17$ must belong to $R_N$, otherwise $h_N$ will be larger than 2.  We first consider the case where $R_N = \left\{1,5\right\}$ or $\left\{5\right\}$.  First, if $N \in S$, then we may choose $p \equiv 5$ (mod 24) and take $q$ to be any other prime divisor, and we have $h_{pq} = 2 = h_N$.  On the other hand, if $N \notin S$ then we must have $5 \mid N$.  Letting $q$ be any divisor of $N$ distinct from 5, $5q \notin S$ and $h_{5q} = 2 = h_N$.  Next suppose $5 \notin R_N$.  This implies $13 \in R_N$, as if it were not we would have $h_N = 4$.  Hence, we may choose divisors $p \equiv 17$ and $q \equiv 13$ (mod 24), which will suffice.  Finally, if $5 \in R_N$ and $R_N \ne \left\{5\right\}$ or $\left\{1,5\right\}$, then taking $p \equiv 5$ and $q \equiv 13$ or $17$ will work.  \\

Lastly, we consider the case where $h_N = 1$.  As $2 \nmid h_N$, there must exist a divisor $p$ which is congruent to $3$ (mod 4).  Additionally, as $3 \nmid h_N$, there must exist some prime divisor $q$ which is congruent to either $3$ or $5$ (mod 6).  Our choice of $p$ guarantees that $2 \nmid h_{pq}$, and our choice of $q$ is such that $3 \nmid h_{pq}$.  Therefore, as $h_{pq}$ is a divisor of 24 which is divisible by neither 2 nor 3, it follows that $h_{pq} = 1$.  Additionally, since $p \equiv 3$ (mod 4), $pq \notin S$.\\

\end{proof}

We now have enough to prove the specialization of Theorem \ref{exist} to the case where $N$ is squarefree.
\begin{prop}\label{sqrfrexist}
Let $N = p_1p_2\cdots p_\ell$ be squarefree, and let $k \geq 2$ be even.  Define $h_N = \frac{1}{2}\gcd\left\{p_1-1,p_2-1,\cdots,p_\ell-1,24\right\}$.  Then there exists an $\eta$-quotient in $M_k(\Gamma_1(N))$ if
\begin{enumerate}[label=(\roman*)]
 \item $h_N \mid k$
 \item It is not the case that $k=2$ and $N \in S$.
\end{enumerate}
\end{prop}
\begin{proof}
Let $N$ be as in the hypotheses and let $k\geq 2$.  We wish to show that there exists an $\eta$-quotient in $M_k(\Gamma_1(N))$ whenever $h_N \mid k$ and either $N \notin S$ or $N \in S$ and $k \geq 4$.  If $\ell = 1$ or $2$, this proposition reduces to Theorem \ref{primeexist} or \ref{semiexist} respectively, and so we need only consider the case where $\ell \geq 3$.  By Lemma \ref{reduce}, there exists a product of two distinct primes $pq \mid N$ such that $h_{pq} = h_N$, and such that if $N \notin S$ then neither is $pq$.  But then, as $h_{pq} = h_N \mid k$, it follows from Theorem \ref{semiexist} that there exists an $\eta$-quotient $f(\tau) \in M_k(\Gamma_1(pq))$.  As $pq \mid N$, $M_k(\Gamma_1(pq)) \subset M_k(\Gamma_1(N))$, and so in particular $f(\tau) \in M_k(\Gamma_1(N))$.
\end{proof}

Theorem \ref{exist} now follows easily.
\begin{proof}[Proof of Theorem \ref{exist}]
 Let $N = p_1^{e_1}p_2^{e_2}\cdots p_{\ell}^{e_\ell}$, and define $h_N$ as above.  We want to show that there exists an $\eta$-quotient $f(\tau) \in M_k(\Gamma_1(N))$ whenever $h_N \mid k$ and it is not the case that $k=2$ and $N \in S$.  Define $\widehat{N} = p_1p_2\cdots p_{\ell}$.  As membership in $S$ is determined entirely by the set of prime divisors of $N$, we have that $\widehat{N} \in S$ if and only if $N \in S$.  Additionally, $h_N = h_{\widehat{N}}$.  By Proposition \ref{sqrfrexist}, it then follows that there exists an $\eta$-quotient $f(\tau) \in M_k(\Gamma_1(\widehat{N}))$ whenever $h_N \mid k$ and it is not the case that $k=2$ and $N \in S$.  The fact that $M_k(\Gamma_1(\widehat{N})) \subset M_k(\Gamma_1(N))$ completes the proof.
\end{proof}
\section{Modular form spaces which do not contain $\eta$-quotients}
In \cite{AAHOS}, the author et al. use the bound from Theorem \ref{RWB} to show the converse to Theorem \ref{exist} in the case where $N$ is a prime or a product of two distinct primes.  We now investigate how well this converse generalizes to arbitrary squarefree levels.  Since holomorphic $\eta$-quotients are also weakly holomorphic, Theorem \ref{weakexist} shows that there can be no $\eta$-quotient in $M_k(\Gamma_1(N))$ whenever $h_N \nmid k$.  The impact that membership in $S$ has on existence of $\eta$-quotients in modular form spaces will be partially answered by Theorem \ref{sqrfrnonexist}.  In order to prove Theorem \ref{sqrfrnonexist}, we must first make some observations about the structure of $\eta$-quotients and their orders of vanishing using Theorem \ref{1.65}.
\begin{defn}
Let $N = p_1p_2\cdots p_\ell$ be coprime to 6, and let $f(\tau) \in M_k(\Gamma_1(N))$ be an $\eta$-quotient.  Note that $N$ has $2^\ell$ divisors, given by the $2^\ell$ ways we can choose to include or exclude each $p_i$.  Using Theorem \ref{1.65} and (\ref{cuspreps}), we have the following system of $2^\ell$ equations associated to the orders of vanishing of $f$ at each cusp of $\Gamma_0(N)$:
\[\left\{ 24v_{\frac{1}{d}} = \sum_{\delta \mid N} \frac{N\gcd(d,\delta)^2}{d\delta}r_\delta\right\}_{d \mid N}.\]
Let $A_N$ be the coefficient matrix for the right-hand side of the above system of equations.  That is, given some ordering $d_1,d_2,\hdots,d_{2^\ell}$ of the divisors of $N$,
\[A_N = \left( \frac{N\gcd(d_i,d_j)^2}{d_id_j}\right)_{1\leq i,j \leq 2^\ell}.\]
\end{defn}
\begin{lemma}\label{sudokulemma}
 Let $N$ be squarefree and coprime to 6, and define $A_N$ as above.  Every $c \mid N$ appears exactly once in each column and row of $A_N$.
\end{lemma}
\begin{proof}
 As the expression $\frac{N\gcd(d_i d_j)^2}{d_id_j}$ is symmetric in $d_i$ and $d_j$, the matrix $A_N$ is symmetric.  Therefore, it suffices to show that $c$ appears exactly once in every row of $A_N$, as we will then know that $c$ appears exactly once in every column of $A_N^T = A_N$.  Additionally, as there are as many rows of $A_N$ as there are divisors or $N$, we need only show that every divisor appears in each row of $A_N$ to know that it appears exactly once.  Hence, we wish to show that, for any fixed $1 \leq i \leq 2^\ell$, there exists $d_j$ such that
\begin{equation}\label{sudoku}
 c = \frac{N\gcd(d_i,d_j)^2}{d_id_j}.
\end{equation}
Observe that $\gcd(c,d_i)\gcd(N/c,N/d_i)$ is a divisor of $N$, so there exists $1 \leq j \leq 2^{\ell}$ such that
\[d_j = \gcd(c,d_i)\gcd(N/c,N/d_i).\]
We will show that (\ref{sudoku}) is satisfied by this choice of $j$, which will finish the proof. \\

Because $N$ is squarefree, so too is any divisor, and so every divisor is a product of distinct primes.  Thus, in checking that various divisors of $N$ are equal, we need only check that they share the same prime divisors.  For example, our choice of $d_j$ can be expressed as the product of the primes dividing both $c$ and $d_i$ and of primes dividing $N$ but dividing neither $c$ nor $d_i$.  That is,
\[d_j = \!\!\!\!\!\!\! \prod_{\substack{p \text{ prime} \\ p \mid c \text{ and } p \mid d_i}} \!\!\!\!\!\!p \quad \times \!\! \prod_{\substack{p \mid N \text{ prime} \\ p \nmid c \text{ and } p \nmid d_i}} \!\!\!\!\!\!p.\]
The primes dividing both $d_i$ and $d_j$ are exactly those appearing in the first term of the above representation of $d_j$. Thus,
\[\gcd(d_i,d_j) = \!\!\!\!\!\!\!\!\prod_{\substack{p \text{ prime} \\ p \mid c \text{ and } p \mid d_i}} \!\!\!\!\!\!p = \gcd(c,d_i).\]
We can hence reduce the right-hand side of (\ref{sudoku}) by
\begin{equation}\label{sudoku2}
 \frac{N\gcd(d_i,d_j)^2}{d_id_j} = \frac{N\gcd(c,d_i)}{d_i\gcd(N/c,N/d_i)}.
\end{equation}
 In the denominator of the right-hand side of (\ref{sudoku2}), we have the term $d_i\gcd(N/c,N/d_i)$, which is the product of all primes which divide $d_i$ and all primes dividing $N$ but dividing neither $c$ nor $d_i$.  That is,
 \begin{equation}\label{sudoku3}
  d_i\gcd(N/c,N/d_i) = \!\!\!\!\!\!\prod_{\text{prime } p \mid d_i} \!\!\!\!p \hspace{2mm} \times \hspace{2mm} \!\!\!\!\!\!\!\prod_{\substack{p \mid N \text{ prime} \\ p \nmid c \text{ and } p \nmid d_i }} \!\!\!\!\!\!p.
 \end{equation}
 So, $N/(d_i\gcd(N/c,N/d_i)$ is the product of all prime divisors of $N$ which do not appear in (\ref{sudoku3}).  As we are removing every prime which divides $d_i$ and every prime which divides neither $c$ nor $d_i$, we are left only with the primes dividing $c$ but not $d_i$.  That is,
 \[\frac{N}{d_i\gcd(N/c,N/d_i)} = \!\!\!\!\!\!\!\!\prod_{\substack{p \mid N \text{ prime} \\ p \mid c \text{ and } p \nmid d_i}} \!\!\!\!\!\!p.\]
 On the other hand, $\gcd(c,d_i)$ is the product of all primes dividing both $c$ and $d_i$.  Combining these, we have
 \[\frac{N\gcd(c,d_i)}{d_i\gcd(N/c,N/d_i)} = \!\!\!\!\!\!\!\!\prod_{\substack{ p \text{ prime} \\ p \mid c \text{ and } p \nmid d_i}} \!\!\!\!\!\!p \hspace{2mm} \times \!\!\!\!\prod_{\substack{p \text{ prime} \\ p \mid c \text{ and } p \mid d_i}} \!\!\!\!\!p = \!\!\!\!\prod_{\text{prime } p \mid c} \!\!\!\!p = c,\]
 showing that (\ref{sudoku}) holds.
\end{proof}
Using this result, we are able to show that in order for $M_k(\Gamma_1(N))$ to contain $\eta$-quotients which are truly quotients and not $\eta$-products, then we must have $k$ large enough or $N$ sufficiently composite relative to its size in the sense of Theorem \ref{RWB}.  We make these ideas precise in the following Lemma.
\begin{lemma}\label{etaproduct}
Let $N$ be squarefree and coprime to 6, and let $p_1$ denote the smallest prime dividing $N$.  If
\[2k\!\!\prod_{\substack{p \mid N \\ p \text{ prime}}} \!\!\frac{p+1}{p-1} < p_1+1,\]
then for any $f(\tau) = \prod_{\delta \mid N}\eta^{r_\delta}(\delta\tau) \in M_k(\Gamma_1(N))$, we must have $r_\delta \geq 0$ for all $\delta \mid N$, so $f(\tau)$ is an $\eta$-product.
\end{lemma}
\begin{proof}
Towards contradiction, suppose $r_d < 0$ for some $d \mid N$.  By Theorem \ref{1.65} and Lemma \ref{sudokulemma}, the order of vanishing at the cusp of $\Gamma_0(N)$ represented by $1/d$ takes the form
\[24v_{\frac{1}{d}} = Nr_d + \sum_{\substack{\delta \mid N \\ \delta \ne d}}c_\delta r_\delta,\]
where each $c_\delta \mid N$, $c_\delta \ne N$.
As $r_d < 0$ is an integer, $r_d \leq -1$, and so
\begin{equation}\label{ineq1}
Nr_d \leq -N.
\end{equation}
As each $c_\delta \leq \frac{N}{p_1}$, we also have, for each $d \ne \delta \mid N$,
\begin{equation}\label{ineq2}
 c_\delta r_\delta \leq \frac{N}{p_1}|r_\delta|. 
\end{equation}
Together (\ref{ineq1}) and (\ref{ineq2}) yield the inequality
\begin{equation}\label{vanineq}
 24v_{\frac{1}{d}} \leq -N + \frac{N}{p_1} \sum_{\substack{\delta \mid N \\ \delta \ne d}} |r_\delta|.
\end{equation}
By Theorem \ref{RWB} and the hypothesis of the Lemma,
\[|r_d|+\sum_{\substack{\delta \mid N \\ \delta \ne d}}|r_\delta| \leq 2k \prod_{p \mid N} \left(\frac{p+1}{p-1}\right) < p_1+1.\]
In particular, as $|r_d| \geq 1$ and each $|r_\delta|$ is an integer,
\begin{equation}\label{vanineq2}
 \sum_{\substack{\delta \mid N \\ \delta \ne d}}|r_\delta| \leq p_1-1.
\end{equation}
Combining (\ref{vanineq}) and (\ref{vanineq2}) we have
\[24v_{\frac{1}{d}} \leq -N + \frac{N}{p_1}(p_1-1) = N\left(\frac{p_1-1}{p_1}-1\right) < 0.\]
But this would imply that $f(\tau)$ has a pole at the cusp represented by $\frac{1}{d}$, contradicting the assumption that $f(\tau) \in M_k(\Gamma_1(N))$.  Thus, $r_d \geq 0$ for all $d \mid N$.
\end{proof}
We now prove Theorem \ref{sqrfrnonexist}
\begin{proof}[Proof of Theorem \ref{sqrfrnonexist}]
 Let $N$ be as in the statement of the theorem.  Suppose we had an $\eta$-quotient $f(\tau) = \prod_{\delta \mid N} \eta^{r_\delta}(\delta\tau) \in M_2(\Gamma_1(N))$.  By Lemma \ref{etaproduct}, each $r_\delta \geq 0$.  As $N$ is coprime to 6, as noted in Remark \ref{1.64conv}, $f(\tau)$ must satisfy the congruence conditions in Theorem \ref{1.64}.  That is,
\begin{equation}\label{1.64cong}
 \sum_{\substack{\delta \mid N \\ \delta \equiv 1 \!\!\!\!\pmod{24}}} \!\!\!\!\!\!\!\!r_\delta \hspace{4mm} + \!\!\!\!\! \sum_{\substack{\delta \mid N \\ \delta \equiv 5 \!\!\!\!\pmod{24}}} \!\!\!\!\!\!\!\!5r_\delta \equiv 0 \pmod{24}.
\end{equation}
Additionally, as the weight $k$ is equal to $1/2$ times the sum of the $r_\delta$,
\begin{equation}\label{weightsum}
\sum_{\delta \mid N} r_\delta = 4.
\end{equation}
However, the largest the left-hand side of (\ref{1.64cong}) can be for non-negative $r_\delta$ satisfying (\ref{weightsum}) is 20, and so these have no simultaneous solutions in non-negative integers.  Hence, there is no $\eta$-quotient in $M_2(\Gamma_1(N))$.
\end{proof}

\newpage

\printbibliography

\end{document}